\documentclass[12pt]{elsarticle}

\usepackage{amssymb,amsfonts,amsmath,times}
\usepackage{amsfonts}
\usepackage{amstext}
\usepackage{amssymb}
\usepackage{amsmath}
\usepackage{amsthm}
\usepackage{geometry}

\usepackage{amsthm,amssymb,amsfonts,amsmath,times}
\usepackage{hyperref}
\usepackage{mathrsfs}
\theoremstyle{plain}
\newtheorem{theorem}{Theorem}[section]
\newtheorem{proposition}[theorem]{Proposition}
\newtheorem{lemma}[theorem]{Lemma}
\newtheorem{corollary}[theorem]{Corollary}

\newtheorem{condition}[theorem]{Condition}

\theoremstyle{remark}
\newtheorem{remark}{Remark}[section]

\theoremstyle{definition}

\begin{document}

\begin{frontmatter}

\title{\normalfont\LARGE
$L^{2}$ dimensions of spaces of braid-invariant harmonic forms}


\author[York]{Alexei Daletskii \corref{cor1}}
\ead{ad557@york.ac.uk}
\author[Kiev]{Alexander Kalyuzhnyi}
\ead{kalyuz@imath.kiev.ua}

\address[York]{Department of Mathematics, University of York, UK}
\address[Kiev]{Institute of Mathematics, NAS, Kiev, Ukraine}

\cortext[cor1]{Corresponding author.}

\begin{abstract}
Let $X$ be a Riemannian manifold endowed with a co-compact isometric
action of an infinite discrete group. We consider $L^{2}$ spaces of harmonic vector-valued forms on the product manifold $X^{N}$, which are invariant with respect to an action of the braid
group $B_N$, and
compute their von Neumann dimensions (\textit{the braided }$L^{2}$-\textit{\
Betti numbers}).
\end{abstract}

\begin{keyword}
\def\MSC{\par\leavevmode\hbox {\it 2010 MSC:\ }}
\def\SC{\par\leavevmode\hbox {\it Subject classification:\ }}
Braid group, von Neumann algebra, $L^{2}$-Betti number, Configuration space

\MSC Primary 58A10, 20F36, 46L10; Secondary 58A12, 58J22
\SC Global Analysis, Analysis on Manifolds
\end{keyword}

\end{frontmatter}

\section{Introduction}

In recent years, there has been a growing interest in the analysis and geometry on
various spaces of configurations (i.e. finite or infinite countable subsets)
in a given topological space $X$, motivated by widespread applications in
various parts of mathematics, including topology, representation theory,
statistical mechanics, quantum field theory, mathematical biology, etc. (see
\cite{Fad}, \cite{VGG}, \cite{GGPS}, \cite{Ge79}, \cite{I}, \cite{AKR3} and
an extensive bibliography in \cite{DVJ1}, \cite{AKR2} and \cite{Rock}). The
development of the general mathematical framework for suitable classes of
configurations is an important and challenging research theme, even in the
case of finite number of components (particles) in each configuration.
Many properties of the spaces of infinite configurations are in turn studied
via the limit transition in the number of particles.

Observe that functions on the space of $N$-particle configurations in $X$ can
be identified with the functions on the Cartesian product $X^{N}$ of $N$
copies of $X$, which are invariant with respect to the natural action of the
symmetric group $S_{N}$ by permutations of arguments. Let us also note that
certain spaces of mappings on $X^{N}$ (e.g. vector-valued functions and
differential forms in the case where $X$ possesses a manifold structure) may
admit more complicated groups of symmetries. An important example of an
extended symmetry group is given by the braid group $B_{N}$, and the spaces
of $B_{N}$-invariant differential forms on $X^{N}$ can be regarded as spaces
differential forms on "\textit{braided configuration spaces}", which are
related to the study of multi-particle systems obeying anyon and plekton
statistics (\cite{Goldin}, \cite{Dell'Antonio}, \cite{GoldinMajid}).

An important and intriguing question in this framework is the relationship
between the properties associated with the multi-particle structure of these
spaces, and the topology and geometry of the underlying space $X$. In many
interesting cases, $X$ is a non-compact Riemannian manifold of infinite
volume, which makes it natural to discuss the above questions in the
framework of the theory of $L^{2}$ invariants. This programme has been
started in \cite{DaS}, \cite{DKal}, \cite{DKal1}, \cite{AD}.

Initiated by M. Atiyah, the theory of $L^{2}$ invariants of non-compact
manifolds possessing infinite discrete groups of isometries serves as a
brilliant example of application of von Neumann algebras in topology,
geometry and mathematical physics. The initial framework introduced in \cite%
{Ati} is as follows. Let $X$ be a non-compact Riemannian manifold admitting
an infinite discrete group $G$ of isometries such that the quotient $K=X/G$
is compact. Then $G$ acts by isometries on the spaces $L^{2}\Omega ^{m}(X)$
of square-integrable $m$-forms over $X$. The projection $P^{(m)}$ onto the
space $\mathcal{H}^{m}(X)$ of square-integrable harmonic $m$-forms (the $m$%
-th $L^{2}$-cohomology space) commutes with the action of $G$ and thus
belongs to the commutant $M_{m}(X)$ of this action, which is (under certain
conditions) a $\mathrm{II}_{\infty }$-factor. The corresponding von Neumann
trace $\mathrm{Tr}_{M_{m}}P^{(m)}=:\beta _{m}(X)$ gives a regularized
dimension of the space $\mathcal{H}^{m}(X)$ and is called the $L^{2}$-Betti
number of $X$ (or $K$). Since the pioneering work \cite{Ati}, $L^{2}$-Betti
numbers and other $L^{2}$ invariants have been studied by many authors (see
e.g. \cite{Mat} and references given there).

The construction above covers in an obvious way the case of the product
manifold $X^{N}$ endowed with the natural action of the product group $G^{N}$%
. The commutant $M_{m}(X^{N})$ of the corresponding action in the space $%
L^{2}\Omega ^{m}(X^{N})$ contains the von Neumann subalgebra $
\bigoplus_{k_{1}+\dots +k_{N}=m}M_{k_{1}}(X)\otimes
...\otimes M_{k_{N}}(X),
$
and the K\"{u}nneth formula shows that
\begin{equation*}
\beta _{m}(X^{N})=\sum_{k_{1}+\dots +k_{N}=m}\beta _{k_{1}}(X)\dots \beta
_{k_{N}}(X).
\end{equation*}

The situation becomes much more complicated if we consider the spaces of
forms on $X^{N}$ admitting additional symmetries, for instance, invariant
with respect to the action of the symmetric group $S_{N}$ by permutations of
arguments. In this case, the projection $\mathbf{P}_{N}^{(m)}$ onto the
space of $S_{N}$-invariant harmonic forms \textit{does not} in general%
\textit{\ }belong to $M_{m}(X^{N}).$The minimal extension $\left\{
M_{m}(X^{N}),\mathbf{P}_{N}^{(m)}\right\} ^{\prime \prime }$ has been
studied in \cite{DaS}, \cite{DKal}, \cite{DKal1}, \cite{AD}, where the
corresponding trace formulae were derived and applied to the computation of $%
L^{2}$-Betti numbers of finite and infinite configuration spaces.

In the present paper, we consider the space $L_{U}^{2}\Omega
^{m}(X^{N},A_{N})$ of square-integrable $m$-forms on $X^{N}$, which take
values in a finite dimensional Abelian algebra $A_{N}$ and are invariant
with respect to a special representation $U$ of the braid group $B_{N}$. The
representation $U$ is constructed as the tensor product of a given unitary
representation $\pi $ of $B_{N}$ in $A_{N}$ and the natural action $%
\mathfrak{S}$ of $B_{N}$ in $L^{2}\Omega ^{m}(X^{N})$ by permutations.

Let $\mathbf{P}_{N}^{(m)}$ be the projection onto the space $\mathcal{H}%
_{U}^{m}(X^{N},A_{N})$ of harmonic forms in \newline $L_{U}^{2}\Omega
^{m}(X^{N},A_{N})$.
The construction of the (minimal) von Neumann algebra
containing $\mathbf{P}_{N}^{(m)}$ and computation of the trace of $\mathbf{P}%
_{N}^{(m)}$ involves the study of the von Neumann algebra
\begin{equation}
\left\{ M_{m} (X^N)\otimes \mathrm{End}(A_{N}),U\right\} ^{\prime \prime }.
\label{BraidedvonNeumann}
\end{equation}%
We give a description of these von Neumann algebras and study the
space $\mathcal{H} _{U}^{m}(X^{N},A_{N})$
in the situation where $%
A_{N}=\left( \mathbb{C}^{2}\right) ^{\otimes N}$ and the representation $\pi
$ is generated by a $4\times 4$ matrix $C$ satisfying the braid equation%
\begin{equation*}
(C\otimes 1)(1\otimes C)(C\otimes 1)=(1\otimes C)(C\otimes 1)(1\otimes C).
\end{equation*}%
Representations of this type have been extensively studied,
see \cite{Kessel}.

The structure of the paper is as follows. In Sections 2 and 3, we derive
general formulae for the traces on von Neumann algebras of the type (\ref%
{BraidedvonNeumann}). In Sections 4 and 5, we compute the $L^{2}$ dimensions
of the spaces $\mathcal{H}_{U}^{m}(X^{N},A_{N})$. In Section 6, we define
the space $\bigoplus_{N\in \mathbb{N}}\mathcal{H}_{U}^{m}(X^{N},A_{N})$ of "%
\textit{harmonic} $m$\textit{-forms on the braided configuration space}" and
compute the \textit{supertrace} of the corresponding projection (via the
limit transition $N\rightarrow \infty $) in terms of the Euler
characteristic of the underlying manifold $X$ and certain invariant of the
representation $U$.

In what follows we refer to \cite{BrR}, \cite{Takesaki}, \cite{JonesSunder}
for general notions of the theory of von Neumann algebras. We denote by $%
\mathrm{Tr}_{K}$ the faithful normal semifinite trace on a $\mathrm{II}%
_{\infty }$- factor $K$.

\bigskip

\section{Trace formula}

\label{sec1}We start with the review of some well-known facts from the
theory of braid groups that we will need in what follows, see e.g. \cite%
{Artin}, \cite{KesselTuraev}. The braid group $B_{N}$\ is generated by
elements $b_{1},\dots ,b_{N-1}$\ satisfing the commutation relations\textbf{%
\ }%
\begin{eqnarray*}
b_{i}b_{j} &=&b_{j}b_{i},\ \left\vert i-j\right\vert \geq 2, \\
b_{i}b_{j}b_{i} &=&b_{j}b_{i}b_{j},\ \left\vert i-j\right\vert =1.
\end{eqnarray*}%
The symmetric group $S_{N}$ is a quotient group of $B_{N}$ with respect to the
relations $b_{i}^{2}=1$,\newline $i=1,\dots N-1$. We denote by $\phi
:B_{N}\rightarrow S_{N}$ the corresponding canonical homomorphism.

Let us consider the \textit{pure braid} group $B_{N}^{pure}:=\mathrm{Ker~}%
\phi $ (or the group of \textit{color braids}). Define the forgetful
homomorphism $f_{N-1}:B_{N}^{pure}\rightarrow B_{N-1}^{pure}$ which acts on
a pure $N$-braid by forgetting its last string. Observe that the kernel of $%
f_{N-1}$ is isomorphic to the free group $F_{N-1}$. Thus one has the split
exact sequence of groups\textbf{\ }%
\begin{equation}
1\rightarrow F_{N-1}\rightarrow B_{N}^{pure}\rightarrow
B_{N-1}^{pure}\rightarrow 1.  \label{split}
\end{equation}%
Thus the pure braid group can be decomposed into the semidirect product $%
B_{N}^{pure}=F_{N-1}\rtimes B_{N-1}^{pure}$ of the free group $F_{N-1}$ and
the group $B_{N-1}^{pure}$ of pure $N-1$ braids. Iterating the decomposition
procedure one can represent $B_{N}^{pure}$ as the iterated semidirect
product of free groups $F_{1}\simeq \mathbb{Z}$, $F_{2},\ \dots ,F_{N-1}$
(the Artin normal form of $B_{N}^{pure}$). The generators $x_{i,j}$, $1\leq
i\leq j$, of the free group $F_{j}$, $j\leq N-1$, can be represented in the
form
\begin{equation*}
x_{i,j}=b_{j}b_{j-1}\dots b_{i+1}b_{i}^{2}b_{i+1}^{-1}\dots
b_{j-1}^{-1}b_{j}^{-1}.
\end{equation*}%
Therefore the elements $x_{i,j}$, $1\leq i\leq j\leq N-1$, generate $%
B_{N}^{pure}$.

In what follows, we will use the following special representation of the
group $B_{N}$. Let $\mathfrak{S}$ be a unitary representation of the
symmetric group $S_{N}$ in a separable Hilbert space $H$ and $\pi $ be a
unitary representation of $B_{N}$ in a finite dimensional Hilbert space $V$.
Then these representations generates a unitary representation $U$ of the
group $B_{N}$ in Hilbert space $H\otimes V$ by the formula
\begin{equation}
U(b_{i})(\xi \otimes v)=\mathfrak{S}(\phi (b_{i}))\xi \otimes \pi (b_{i})v,
\label{reprBN}
\end{equation}%
where $\xi \in H,\ v\in V$. A direct computation shows that
\begin{equation*}
U(b_{i})U(b_{i+1})U(b_{i})=U(b_{i+1})U(b_{i})U(b_{i+1}),
\end{equation*}%
which implies that formula (\ref{reprBN}) indeed defines a unitary
representation of the group $B_{N}$.

Let us introduce the subspace
\begin{equation}
\left( H\otimes V\right) _{U}=\{f\in H\otimes V|\ U(b_{i})f=f,i=1,\dots
,N-1\}  \label{inv-subs}
\end{equation}%
of $U$-invariant elements of $H\otimes V$ and consider the corresponding
orthogonal projection
\begin{equation*}
P_{U}:H\otimes V\rightarrow \left( H\otimes V\right) _{U}.
\end{equation*}

Let $M$ be a $\mathrm{II}_{\infty }$ factor acting in $H$ and define $%
Q=M\otimes \mathrm{End}(V)$. In general, the projection $P_{U}$ does not in
general belong to $Q$. In what follows, we will extend $Q$ to a suitable $%
\mathrm{II}_{\infty }$ factor containing $P_{U}$. We suppose that the
following condition holds.

\begin{condition}
\label{outer}The operators $\mathfrak{S}(\sigma )$, $\sigma \in S_{N}$, $%
\sigma \neq e$, do not belong to $M$\textbf{.}
\end{condition}

Condition \ref{outer} will be verified in the particular situations
considered in the following sections. It implies that the formula
\begin{equation}
\alpha _{\sigma }(x)=\mathfrak{S}(\sigma )x\mathfrak{S}^{\ast }(\sigma
),\quad x\in M  \label{act}
\end{equation}%
defines a nontrivial outer action
\begin{equation*}
S_{N}\ni \sigma \mapsto \alpha _{\sigma }\in \mathrm{Aut}(M)
\end{equation*}%
of the group $S_{N}$ by automorphisms of the $\mathrm{II}_{\infty }$-factor $%
M$. Consider the corresponding cross-product $M\rtimes _{\alpha }S_{N}$. It
is well known that an automorphism of a factor is free iff it is outer.
Therefore (see e.g. \cite{JonesSunder}, Proposition 1.4.4) $M^{\prime }\cap
(M\rtimes _{\alpha }S_{N})=\mathbb{C}$ so that $M\rtimes _{\alpha }S_{N}$ is
a ($\mathrm{II}_{\infty }$-) factor.

\begin{lemma}
\label{isomorphism}The map
\begin{equation}
i:M\rtimes _{\alpha }S_{N}\ni (x,\sigma )\mapsto xU(\sigma )\in \{M,\{%
\mathfrak{S}(\sigma )\}_{\sigma \in S_{N}}\}^{\prime \prime }  \label{map}
\end{equation}%
is an isomorphism of factors.
\end{lemma}

\begin{proof}
Observe that the map (\ref{map}) is a non-trivial surjective normal
homomorphism. It is well known that factors do not contain proper weakly
closed two-sided ideals, which immediately implies that any homomorphisms
between them is either identically zero or\textbf{\ }injective, and the
result follows.
\end{proof}

Let us define the space
\begin{equation*}
V^{\pi }=\{v\in V\ |\ \pi (\beta )v=v\ \text{for all}\ \beta \in
B_{N}^{pure}\}
\end{equation*}%
of $\pi (B_{N}^{pure})$-invariant elements of $V$.

\begin{lemma}
\label{invarspace}

\begin{itemize}
\item[(i)] The space $V^{\pi }$ is invariant under the representation $\pi $
of the braid group $B_{N}$.

\item[(ii)] The restriction of the representation $\pi $ of $B_{N}$ to the
space $V^{\pi }$ defines a representation $\tilde{\pi}$\ of the symmetric
group $S_{N}$ via the formula
\begin{equation*}
\tilde{\pi}(\varphi (b_{i}))=\pi (b_{i})\upharpoonright V^{\pi },i=1,\dots
,N-1.
\end{equation*}
\end{itemize}
\end{lemma}

\begin{proof}
\textsl{(i)}. It is sufficient to show that the subspace $V^{\pi }$ is
invariant under the action of operators $\pi (b_{i})$, $i=1\dots ,N-1$. Let
us suppose that $V^{\pi }$ is not invariant w.r.t. $\pi (b_{N-1})$. Then
there exists an element $w\in \pi (b_{N-1})V^{\pi }$ orthogonal to $V^{\pi }$%
. Let $w=\pi (b_{N-1})v$, where $v\in V^{\pi }$. The equality $\pi
(b_{N-1}^{2})v=v$ implies that $\pi (b_{N-1})w=v$ and $\pi (b_{N-1}^{2})w=w$%
. Then
\begin{equation*}
v=\pi (b_{N-1})\pi (b_{N-2}^{2})\pi (b_{N-1}^{-1})v=\pi (b_{N-1})\pi
(b_{N-2}^{2})w,
\end{equation*}%
so that we have $\pi (b_{N-2}^{2})w=w$. The relations $%
b_{N-1}b_{k}=b_{k}b_{N-1}$, $k\leq N-3$, and $\pi (x_{i,j})v=v$ imply that
\begin{equation*}
\pi (x_{i,k})w=\pi (x_{i,k})\pi (b_{N-1})v=\pi (b_{N-1})\pi (x_{i,k})v=\pi
(b_{N-1})v=w
\end{equation*}%
for all $1\leq i\leq k\leq N-3$. We have also
\begin{eqnarray*}
\pi (x_{i,N-1})w &=&\pi (b_{N-1})\pi (x_{i,N-2})\pi (b_{N-1}^{-1})w \\
&=&\pi (b_{N-1})\pi (x_{i,N-2})v=\pi (b_{N-1})v=w
\end{eqnarray*}%
because $\pi (b_{N-1}^{-1})w=v$ then for all $i=1,\dots ,N-2$. Finally, let
us show that the vector $w$ is invariant with respect to $\pi (x_{i,N-2})$, $%
i=1,\dots ,N-2$. The equality
\begin{equation*}
\pi (b_{N-1})\pi (x_{i,N-2})\pi (b_{N-1}^{-1})v=v
\end{equation*}%
implies that $\pi (b_{N-1})\pi (x_{i,N-2})w=v$, and
\begin{equation*}
\pi (x_{i,N-2})w=\pi (b_{N-1}^{-1})v=w.
\end{equation*}%
We see that $w$ is invariant with respect to all operators $\pi (x_{i,j})$, $%
1\leq i\leq j\leq N-1$. Elements $x_{i,j}$ generate the pure braid group,
which implies that $w\in V^{\pi }$\textbf{. }Therefore the space $V^{\pi }$
is $\pi (b_{N-1})$-invariant under the operator \newline
By similar arguments one can show that $V^{\pi }$\textbf{\ }is invariant
w.r.t. the action of operators $\pi (b_{k})$\textbf{, }$k=1,\dots N-2$%
\textbf{. }

\textsl{(ii)} follows from the equality $\pi (b_{i}^{2})\upharpoonright ${$%
V^{\pi }$}$=\mathrm{id}_{V^{\pi }}$ and \textsl{(i)}.
\end{proof}

\bigskip

Consider the projection%
\begin{equation*}
p_{\pi }:V\rightarrow V^{\pi }.
\end{equation*}%
We have the following result.

\begin{theorem}
\label{projbraid} The projection $P_{U}:H\otimes V\mapsto \left( H\otimes
V\right) _{U}$ can be represented in the form
\begin{equation}
P_{U}=\frac{1\otimes p_{\pi }}{N!}\sum_{\sigma \in S_{N}}\mathfrak{S}(\sigma
)\otimes \tilde{\pi}(\sigma ),  \label{proj_formulae}
\end{equation}%
where $\tilde{\pi}$\ is the unitary representation of the group $S_{N}$
constructed in Lemma \ref{invarspace}.
\end{theorem}

\begin{proof}
Let us denote by $A$\ the right-hand side of (\ref{proj_formulae}).
\textbf{\ I}t follows from Lemma %
\ref{invarspace} that the relation $\pi (b_{i})p_{\pi }=p_{\pi }\pi (b_{i})$%
\ holds for any $i=1,\dots ,N-1$. Therefore $A\in \mathrm{End\ }(H\otimes V)$%
\ is a projection.

It is sufficient
to show that images of $P$ and $A$ coincide.
It is obvious that if $\xi \in \left( H\otimes V\right)
_{U}$\ then $(1\otimes p_{\pi })\xi =\xi $\ and $\mathfrak{S}_{\sigma
}\otimes \tilde{\pi}(\sigma )\xi =\xi $. Therefore the space $\left(
H\otimes V\right) _{U}=P(H\otimes V)$\ is contained in the image of $A$.
Let us prove the converse inclusion. Let an element $\xi \in H\otimes V$\
belong to the image of $A$. Then $\xi \in H\otimes V^{\pi }$\ and $\mathfrak{%
S(\sigma )}\otimes \tilde{\pi}(\sigma )\xi =\xi $. Since $\pi
(b_{i})\upharpoonright V^{\pi }=\tilde{\pi}(\varphi (b_{i}))$\ we have
\begin{equation*}
U(b_{i})\xi =\mathfrak{S(\varphi (b_{i}))}\otimes \pi (b_{i})\xi =\mathfrak{S%
}_{\varphi (b_{i})}\otimes \tilde{\pi}(\varphi (b_{i}))\xi =\xi ,
\end{equation*}%
i. e. $\xi $\ is $U$-invariant. Thus the image of $A$\ is contained in the
space $\left( H\otimes V\right) _{U}=P_{U}(H\otimes V)$.
\end{proof}

\begin{corollary}
We have the inclusion\textbf{\ }$\left( H\otimes V\right) _{U}\subset
H\otimes V^{\pi }$\textbf{.}
\end{corollary}

Let us introduce the von Neumann algebra $Q_{U}:=\{Q,\{U(\beta )\}_{\beta
\in B_{N}}\}^{\prime \prime }$ generated by $Q$ and the operators $U(\beta
),\beta \in B_{N}$. It is obvious that $P_{U}\in Q_{U}$\ .

\begin{lemma}
\label{braidfactor}Von Neumann algebras $(M\rtimes S_{N})\otimes \mathrm{End}%
(V)\ $and $Q_{U}$ are isomorphic.
\end{lemma}

\begin{proof}
Recall that operators $\mathfrak{S}(\phi (\beta ))$, $\beta \in B_{N}$, $%
\beta \neq e$, do not belong to $M$,\ which implies that operators\textbf{\ }%
\begin{equation*}
U(\beta )=\mathfrak{S}(\phi (\beta ))\otimes \pi (\beta ),\ \beta \in
B_{N},\ \beta \notin B^{pure}_N,
\end{equation*}%
do not belong to $Q$.\textbf{\ }Therefore the representation $U$ generates a
nontrivial outer action $\alpha ^{\prime }$ of the group $S_{N}$ on the $%
\mathrm{II}_{\infty }$ factor $Q$. Applying the arguments similar to the
proof of Lemma \ref{isomorphism} one can see that there exists an
isomorphism
\begin{equation}
(Q\rtimes _{\alpha ^{\prime }}S_{N})\simeq \{Q,\{U(\beta )\}_{\beta \in
B_{N}}\}^{\prime \prime }.  \label{semiQ}
\end{equation}%
On the other hand, for any $\beta \in B_{N}$ we have the inclusion\textbf{\ }%
$\pi (\beta )\in \mathrm{End}(V)$\textbf{, }which implies that the factors $%
Q\rtimes _{\alpha ^{\prime }}S_{N}$\ and $(M\rtimes _{\alpha }S_{N})\otimes
\mathrm{End}(V)$\ are isomorphic, and the result follows.
\end{proof}

\begin{corollary}
\label{trace1 copy(1)} $Q_{U}$ is a $\mathrm{II}_{\infty }$-factor.
\end{corollary}

\begin{proof}
Follows from the fact that$(M\rtimes S_{N})\otimes \mathrm{End}(V)$ is a $%
\mathrm{II}_{\infty }$-factor.
\end{proof}

\begin{corollary}
It follows from (\ref{braidfactor}) that $\mathrm{Tr}_{Q_{U}}=\mathrm{Tr}%
_{M\rtimes S_{N}}\times \mathrm{tr}_{\mathrm{End}(V)}$.
\end{corollary}

Without loss of generality we assume that the trace on $\mathrm{End}(V)$ is
normalized, that is, $\mathrm{tr}_{\mathrm{End}(V)}(1)=1$.

\begin{corollary}
\label{trace11} For any $A\in M$\textrm{\ }and $B\in \mathrm{End}\ (V)$
commuting with $p_{\pi }$ we have
\begin{equation}
\mathrm{Tr}_{Q_{U}}((A\otimes B)P_{U})=\frac{\mathrm{Tr}_{M}(A)}{N!}\mathrm{%
tr}_{\mathrm{End}(V)}(Bp_{\pi }).  \label{tr_braid0}
\end{equation}
\end{corollary}

\begin{proof}
It follows the properties of the trace on a semidirect product of factors
that for any $\sigma \in S_{N}$ the following equality holds:
\begin{eqnarray*}
\mathrm{Tr}_{Q_{U}}(A\mathfrak{S}(\sigma )\otimes Bp_{\pi }\tilde{\pi}%
(\sigma )) &=&\mathrm{Tr}_{M\rtimes S_{N}}(A\mathfrak{S}(\sigma ))\mathrm{tr}%
_{\mathrm{End}(V)}(Bp_{\pi }\tilde{\pi}(\sigma \mathbf{))} \\
&=&\delta _{e,\sigma }\mathrm{Tr}_{M}(A)\mathrm{tr}_{\mathrm{End}%
(V)}(Bp_{\pi }\tilde{\pi}(\sigma \mathbf{))},
\end{eqnarray*}%
where $\delta _{e,\sigma }$ is the Kronecker symbol. Formula (\ref{tr_braid0}%
) follows now from (\ref{proj_formulae}).
\end{proof}

\begin{remark}
Introduce the projection%
\begin{equation*}
P_{s}=\frac{1}{N!}\sum_{\sigma \in S_{N}}\mathfrak{S}_{\sigma }.
\end{equation*}%
It follows from (\ref{tr_braid0}) that the following formula holds:
\begin{eqnarray*}
\mathrm{Tr}_{Q_{U}}((A\otimes B)P_{U})=\mathrm{Tr}_{M}(P_{s})\mathrm{Tr}%
_{Q_{U}}((A\otimes B)(1\otimes p_{\pi })) && \\
=\frac{1}{N!}\mathrm{Tr}_{Q_{U}}((A\otimes B)(1\otimes p_{\pi })).\ &&
\end{eqnarray*}
\end{remark}

In order to extend formula (\ref{tr_braid0}) to general elements of $Q_{U}$,
we introduce $\mathrm{II}_{\infty }$-factors
\begin{equation*}
Q^{1}:=M\otimes \mathrm{End}\ (V^{\pi })\text{ and }Q^{2}:=M\otimes \mathrm{%
End}\ ((V^{\pi })^{\perp }).
\end{equation*}%
It is well known that any finite dimensional unitary representation of a
locally compact group is completely reducible. Thus the representation $\pi $
can be decomposed into a direct sum $\pi =\pi ^{1}\oplus \pi ^{2}$ where $%
\pi ^{1}=\tilde{\pi}=\pi \upharpoonright V^{\pi }$ and $\pi ^{2}=\pi
\upharpoonright (V^{\pi })^{\perp }$. Observe that in the case when $\pi $
is irreducible we have $V^{\pi }=\{0\}$ and $\pi ^{1}$ is a zero
representation. Therefore we can introduce von Neumann algebras
\begin{equation*}
\tilde{Q}^{1}=\{Q^{1},\{\mathfrak{S}(\sigma )\otimes \tilde{\pi}(\sigma
)\}_{\sigma \in S_{N}}\}^{\prime \prime }\text{ and }\tilde{Q}^{2}=\{Q^{2},\{%
\mathfrak{S(\phi (\beta ))}\otimes \pi ^{2}(\beta )\}_{\beta \in
B_{N}}\}^{\prime \prime }.
\end{equation*}%
Clearly $\tilde{Q}^{1}$ and $\tilde{Q}^{2}$ are isomorphic to $\mathrm{II}%
_{\infty }$-factors $Q^{1}\rtimes S_{N}$ and $Q^{2}\rtimes S_{N}$
respectively, and $P_U\in \tilde{Q}^{1}$.

We define the following natural trace on the von Neumann algebra $\tilde{Q}%
^{1}\oplus \tilde{Q}^{2}$ setting
\begin{equation}
\mathrm{Tr}_{\tilde{Q}^{1}\oplus \tilde{Q}^{2}}(A_{1}\oplus A_{2})=\frac{%
\mathrm{dim}\ V^{\pi }}{\mathrm{dim}\ V}\mathrm{Tr}_{\tilde{Q}^{1}}(A_{1})+%
\frac{\mathrm{dim}\ (V^{\pi })^{\perp }}{\mathrm{dim}\ V}\mathrm{Tr}_{\tilde{%
Q}^{2}}(A_{2}),
\end{equation}%
$A_{1}\in \tilde{Q}^{1}$, $A_{2}\in \tilde{Q}^{2}$.

Consider an arbitrary $A\in Q$ commuting with $1\otimes
p_{\pi }$. Then $A\in Q^{1}\oplus Q^{2}$ and $AP_{U}\in \tilde{Q}^{1}$.
Observe that $\frac{\mathrm{dim}\ V^{\pi }}{\mathrm{dim}\ V}=\mathrm{tr}_{%
\mathrm{End}(V)}(p_{\pi })$ and $A(1\otimes p_{\pi })\in Q^{1}$. Applying
the arguments similar to the proof of Corollary \ref{trace11} one can see
that%
\begin{equation}
\mathrm{Tr}_{\tilde{Q}^{1}\oplus \tilde{Q}^{2}}(AP_{U})=\frac{\mathrm{Tr}%
_{Q^{1}}(A(1\otimes p_{\pi }))}{N!}\mathrm{tr}_{\mathrm{End}(V)}(p_{\pi }).
\label{tr_braid_reduce0}
\end{equation}%

\begin{remark}\label{trace_reduce}
Note that formula (\ref{tr_braid_reduce0}) is compatible with (\ref%
{tr_braid0}). Indeed, if $A\otimes B\in M\otimes \mathrm{End}\ (V)$ and $%
B\in \mathrm{End}\ (V)$ commutes with $p_{\pi }$, then $A\otimes B\in
Q^{1}\oplus Q^{2}$ and by (\ref{tr_braid_reduce0}) we have
\begin{eqnarray*}
\mathrm{Tr}_{\tilde{Q}^{1}\oplus \tilde{Q}^{2}}((A\otimes B)P_{U}) &=&\frac{%
\mathrm{Tr}_{Q^{1}}(A\otimes Bp_{\pi })}{N!}\mathrm{tr}_{\mathrm{End}%
(V)}(p_{\pi }) \\
&=&\frac{\mathrm{Tr}_{M}(A)}{N!}\mathrm{tr}_{\mathrm{End}(V^{\pi })}(Bp_{\pi
})\mathrm{tr}_{\mathrm{End}(V)}(p_{\pi }) \\
&=&\frac{\mathrm{Tr}_{M}(A)}{N!}\mathrm{tr}_{\mathrm{End}(V)}(Bp_{\pi })
\end{eqnarray*}%
because the traces $\mathrm{tr}_{\mathrm{End}(V)}$ and $\mathrm{tr}_{\mathrm{%
End}(V^{\pi })}$ are normalized.
\end{remark}

\section{Braided multi-particle spaces and the corresponding trace}

\label{sec2}In this section we apply the general construction discussed
above to the case where the space $H$ has the form of the tensor product $%
H_{0}^{\otimes N}$ of $N$ copies of a separable Hilbert space $H_{0}$ and
\begin{equation}
S_{N}\ni \sigma \mapsto \mathfrak{S}(\sigma )\in \mathrm{End}(H_{0}^{\otimes
N})  \label{permut}
\end{equation}%
is the natural action of $S_{N}$ in $H_{0}^{\otimes N}$ by permutations, so
that $\mathfrak{S}(\phi (b_{i}))$ is the permutation of the $i$-th and $i+1$%
-th components in $H_{0}^{\otimes N}$. As in the previous section, the
representation $U$ is defined by formula (\ref{reprBN}), that is,
\begin{equation}
U(b_{i})(\xi \otimes v)=\mathfrak{S}(\phi (b_{i}))\xi \otimes \pi (b_{i})v,
\label{reprBNtensor}
\end{equation}%
where $\xi \in H,\ v\in V$, elements $b_{1},\dots ,b_{N-1}$ are the
generators of the group $B_{N}$ and $\phi :B_{N}\rightarrow S_{N}$ is the
canonical homomorphism. The corresponding subspace $\left( H\otimes V\right)
_{U}$ of invariant elements (cf. (\ref{inv-subs}))\textbf{\ }will be called
the \textit{braided} $N$\textit{-particle space. }

Let $M_{0}$ be a $\mathrm{II}_{\infty }$ factor acting in $H_{0}$ and define
the corresponding $\mathrm{II}_{\infty }$ factors $M=M_{0}^{\otimes N}$ and $%
Q=M\otimes \mathrm{End}\ V$ acting in $H=H_{0}^{\otimes N}$ and $H\otimes V$.

The following result was proved in \cite{DKal} (see also \cite{DKal1}).

\begin{proposition}
\label{condition1}Representation $\mathfrak{S}$ defined by formula (\ref%
{permut}) satisfies Condition \ref{outer}, that is, $\mathfrak{S}(\sigma
)\notin M$ for all $\sigma \in S_{N}$.
\end{proposition}

Proposition \ref{condition1} implies that we can apply the theory developed
in the previous section. According to Lemma \ref{braidfactor} we have
\begin{equation}
P_{U}\in Q_{U}=(M_{0}^{\otimes N}\times _{\alpha }S_{N})\otimes \mathrm{End}%
(V).  \label{isom}
\end{equation}%
In what follows we will show that $Q_{U}$ is the minimal extension of $%
M_{0}^{\otimes N}$\textbf{\ }containing\textbf{\ }and $P_U$\textbf{.}

\begin{theorem}
\label{equal} Factors $Q_{U}$ and $Q_{P}:=\{Q,P_U\}^{\prime \prime }$ are
isomorphic.
\end{theorem}

\begin{proof}
The inclusion $Q_{P}\subset (M_{0}^{\otimes N}\times _{\alpha }S_{N})\otimes
\mathrm{End}(V)$ follows from (\ref{isom}).

Let us prove the inverse inclusion. Observe that operators $\pi (b_{i})\in
\mathrm{End}(V)$ and hence $1_{M_{0}^{\otimes N}\times _{\alpha
}S_{N}}\otimes \pi (b_{i})\in Q_{P}$. This it suffices to prove that the
operators $\mathfrak{S}(\phi (b_{i}))\otimes 1_{\mathrm{End}(V)}$ belong to $%
Q_{P}$. Denote $P_{i}=\frac{1}{2}(1+\mathfrak{S}(\phi (b_{i}))\otimes 1_{%
\mathrm{End}(V)})$. \textbf{\ }The $\mathrm{II}_{\infty }$-factor $M$\ can
be represented in a form $N\otimes B(l_{2})$\textbf{,} where $N$ is a $%
\mathrm{II}_{1}$ factor and $B(l_{2})$ is a space of all bounded operators
in the Hilbert space $l_{2}=l_{2}(\mathbb{N)}$. Therefore $M$ contains the
isometry $v=1_{N}\otimes T$, where $Te_{i}=e_{i+1}$ is the operator of
unilateral shift in $l_{2}$. It is clear that $(v^{\ast })^{m}\rightarrow 0$%
, $m\rightarrow \infty $, strongly. Therefore we have strong convergence
\begin{align}
& \left( (1\otimes 1\otimes (v^{\ast })^{m}\otimes \dots \otimes (v^{\ast
})^{m})\otimes 1_{\mathrm{End}(V)}\right) \cdot \\
& \cdot P\left( (1\otimes 1\otimes v^{m}\otimes \dots \otimes v^{m})\otimes
1_{\mathrm{End}(V)}\right) \rightarrow \frac{2}{N!}P_{1},  \notag
\end{align}%
$m\rightarrow \infty $. Thus $P_{1}\in M$, which implies that $P_{1}\in
M\otimes \mathrm{End}(V)\subset Q_{P}$. Similar arguments show that $%
P_{i}\in Q_{P}$ for any $i=1,\dots ,N-1$.
\end{proof}

\begin{remark}
It follows from Theorem \ref{equal} that formulae (\ref{tr_braid0}) and (\ref%
{tr_braid_reduce0}) are valid for $\mathrm{Tr}_{Q_{P}}$ instead of $\mathrm{%
Tr}_{Q_{U}}$.\textrm{\ }
\end{remark}

\section{Von Neumann dimensions of the spaces of braid-invariant harmonic
forms}

\label{sec3}

Let $X$ be a smooth connected Riemannian manifold admitting an infinite
discrete group $G$ of isometries such that the quotient $K=X/G$ is a compact
connected Riemannian manifold.

In this section we consider the De Rahm complex of square-integrable forms
on $X^{N}$ taking values in a finite dimensional Abelian algebra $A_{N}$
\cite{Guichardet}. We apply the trace formula (\ref{tr_braid0}) in order to
find the von Neumann dimensions of the spaces of $U$-invariant harmonic
forms (the \textit{braided }$L^{2}$-\textit{\ Betti numbers}).

\subsection{Setting: Von Neumann algebras associated with infinite coverings
of compact manifolds \label{setting}}

Let us describe the framework introduced by M. Atiyah in his theory of $%
L^{2} $-Betti numbers, which we will use during the rest of the paper. For a
detailed exposition, see \cite{Ati} and e.g. \cite{Mat}.

We assume that there exists an infinite discrete group $G$ acting freely on $%
X$ by isometries and that $K=X/G$ is a compact Riemannian manifold. That is,
\begin{equation}
G\rightarrow X\rightarrow K  \label{cover}
\end{equation}%
is a Galois (normal) cover of $K$.

Throughout this section, we fix $m=1,...,d,\;d=\dim X-1,$ and use the
following general notations: \newline
$L^{2}\Omega ^{m}(X)$ - the space of square-integrable $m$-forms on $X$;%
\newline
$L^{2}\Omega ^{m}(K)$ - the space of square-integrable $m$-forms on $K$;%
\newline
$H_{X}^{(m)}$ - Hodge-de Rahm Laplacian on $X$ (considered as a self-adjoint
operator in $L^{2}\Omega ^{m}(X)$);\newline
$\mathcal{H}^{m}(X):=\mathrm{Ker}\,H_{X}^{(m)}$ - the space of
square-integrable harmonic $m$-forms on $X$.

The action of $G$ in $X$ generates in the natural way the unitary
representation of $G$ in $L^{2}\Omega ^{m}(X)$, which we denote
\begin{equation}
G\ni g\mapsto T_{g}\in \mathrm{End}(L^{2}\Omega ^{m}(X)).
\end{equation}%
Let $M_{m}$ be the commutant of this representation,
\begin{equation}
M_{m}=\left\{ T_{g}\right\} _{g\in G}^{\prime }\subset \mathrm{End}%
(L^{2}\Omega ^{m}(X)).
\end{equation}%
It is clear that the space $L^{2}\Omega ^{m}(X)$ can be described in the
following way:
\begin{equation}
L^{2}\Omega ^{m}(X)=L^{2}\Omega ^{m}(K)\otimes l^{2}(G),  \label{rep-x}
\end{equation}%
with the group action obtaining the form
\begin{equation}
T_{g}=id\otimes L_{g},
\end{equation}%
$g\in G,$ where $L_{g},\;g\in G,$ are operators of the left regular
representation of $G$. Then
\begin{equation}
\mathrm{End}(L^{2}\Omega ^{m}(X))=\mathrm{End}(L^{2}\Omega ^{m}(K))\otimes
\mathrm{End}(l^{2}(G))
\end{equation}%
and
\begin{equation}
M_{m}=\mathrm{End}(L^{2}\Omega ^{m}(K))\otimes \mathcal{R}(G),  \label{rep-a}
\end{equation}%
where $\mathcal{R}(G)$ is the von Neumann algebra generated by the right
regular representation of $G$.

In what follows, we assume that $G$ is an ICC group, that is,
\begin{equation}
\text{\textit{all\thinspace non-trivial classes\thinspace of\thinspace
conjugate\thinspace elements\thinspace are\thinspace infinite.}}
\label{cond-naim}
\end{equation}

This ensures that $\mathcal{R}(G)$ is a $\mathrm{II}_{1}$-factor (see e.g.
\cite{Takesaki}). Thus $M$ is a $\mathrm{II}_{\infty }$-factor.

Let us consider the orthogonal projection
\begin{equation}
P_{Har}^{(m)}:L^{2}\Omega ^{m}(X)\rightarrow \mathcal{H}^{m}(X)
\end{equation}%
and its integral kernel
\begin{equation}
k(x,y)\in T_{x}X\otimes T_{y}X.
\end{equation}%
Then, because of the $G$-invariance of the Laplacian $H_{X}^{(m)}$, we have $%
P_{Har}\in M_{m}$. It was shown in \cite{Ati} that
\begin{equation}
\mathrm{Tr}_{M}P_{Har}^{(m)}=\int_{K}\mathrm{tr\,}k(x,x)\,dx,
\end{equation}%
where $\mathrm{tr}$ is the usual matrix trace and $dx$ is the Riemannian
volume on $K$. Let us remark that, because of $G$-invariance, $k(m,m)$ is a
well-defined function on $K$. Moreover, it is known that $H_{X}^{(m)}$ is
elliptic regular, which implies that the kernel $k$ is smooth. Thus
\begin{equation}
\beta _{m}:=\mathrm{Tr}_{M}\;P_{Har}^{(m)}<\infty .
\end{equation}%
The numbers $\beta _{m},\;m=0,1,...,d,$ are called the $L^{2}$-Betti numbers
of $X$ (or $K$) associated with $G$. $L^{2}$-Betti numbers are homotopy
invariants of $K$ (\cite{Dod}). It is known that

\begin{equation}
\sum_{m=0}^{d}(-1)^{m}\beta _{m}=\chi (K),
\end{equation}%
where $\chi (K)$ is the Euler characteristic of $K$ ( $L^{2}$ index theorem
\cite{Ati}).

Observe that all of the above can be applied to the product group $G^{N}$
acting on the product manifolds $X^{N}$ (instead of $G$ acting on $X$). The K%
\"{u}nneth formula
\begin{equation}
\mathcal{H}^{m}(X^{N})=\bigoplus_{m_{1}+\dots +m_{N}=m}\ \mathcal{H}%
^{m_{1}}(X)\otimes \dots \otimes \mathcal{H}^{m_{N}}(X).
\end{equation}%
implies that the corresponding $L^{2}$-Betti numbers of $X^{N}$ have the form%
\begin{equation}
\beta _{m}(X^{N})=\sum_{m_{1}+\dots +m_{N}=m}\beta _{m_{1}}\dots \beta
_{m_{N}}.  \label{XN-Betti}
\end{equation}

\subsection{$L^{2}$-dimensions of the spaces of braid-invariant harmonic forms}

Let $A_{N}$ be a finite dimensional Abelian algebra possessing a structure
of a finite dimensional Hilbert space. Consider the Hilbert space $%
L^{2}\Omega ^{m}(X^{N})\otimes A_{N}$ of square integrable (with respect to
the volume measure) $A_{N}$-valued $m$-forms on $X^{N}$.

Recall that the group $G$ acts on $X$ by isometries. Denote by
\begin{equation*}
M_{m}(X^{N})\subset \mathrm{End}(L^{2}\Omega ^{m}(X^{N}))
\end{equation*}%
and
\begin{equation*}
Q^{(m,N)}:=M_{m}(X^{N})\otimes \mathrm{End}(A_{N})\subset \mathrm{End}%
(L^{2}\Omega ^{m}(X^{N})\otimes A_{N})
\end{equation*}%
the commutant of the induced natural action of $G^{N}$ (by isometries) on $%
L^{2}\Omega ^{m}(X^{N})$ and $L^{2}\Omega ^{m}(X^{N})\otimes A_{N}$,
respectively. Here $\mathrm{End}(A_{N})$ is the space of endomorphisms of $%
A_{N}$ with respect to its Hilbert space structure (but not in general morphisms of
the algebra $A_{N}$). Both $M_{m}(X^{N})$\ and $Q^{(m,N)}$ are $\mathrm{II}%
_{\infty }$\ factors.

We adopt the construction of Section \ref{sec1} with $H=L^{2}\Omega
^{m}(X^{N})$ and $V=A_{N}$. Consider a unitary representation
\begin{equation*}
\pi :B_{N}\ni b\mapsto \pi (b)\in \mathrm{End}(A_{N})
\end{equation*}%
of the braid group $B_{N}$ in $A_{N}$, and the corresponding
representation (cf. (\ref{reprBN}))
\begin{equation*}
U:B_{N}\ni b\mapsto U(b)=\mathfrak{S(}\phi (b)\mathfrak{)}\otimes \pi (b)\in
\mathrm{End}(L^{2}\Omega ^{m}(X^{N})\otimes A_{N})
\end{equation*}%
in the space $L^{2}\Omega ^{m}(X^{N})\otimes A_{N}$, where $\mathfrak{S}$ is
the natural action the group $S_{N}$ in $L^{2}\Omega ^{q}(X^{N})$ defined in
local coordinates by the formula\textbf{\ }%
\begin{equation}
\mathfrak{S}(\sigma )f((x_{m})_{m=1,...N})dx_{i_{1}}\wedge \dots \wedge
dx_{i_{q}}=f((x_{\sigma (m)})_{m=1,...N})dx_{\sigma (i_{1})}\wedge \dots
\wedge dx_{\sigma (i_{m})}.  \label{permut1}
\end{equation}%
Here $\sigma \in S_{N}$ is the permutation $(1,...,N)\mapsto (\sigma
(1),...,\sigma (N))$. \textbf{\ }Let
\begin{equation*}
L^{2}\Omega _{U}^{m}(X^{N},A_{N}):=\left( L^{2}\Omega ^{m}(X^{N})\otimes
A_{N}\right) _{U}
\end{equation*}%
be the Hilbert space of $U$-invariant elements of $L^{2}\Omega
^{m}(X^{N})\otimes A_{N}$ (cf. (\ref{inv-subs})), and consider the
corresponding projection
\begin{equation*}
P_{U}^{(m)}:L^{2}\Omega ^{m}(X^{N})\otimes A_{N}\rightarrow L^{2}\Omega
_{U}^{m}(X^{N},A_{N}).
\end{equation*}

Let $A_{N}^{\pi }:=\left\{ a\in A_{N}:\pi (b)a=a,~b\in B_{N}^{pure}\right\} $
be the space of all $\pi (B_{N}^{pure})$-invariant elements of $A_{N}$. In
what follows we will assume that
\begin{equation}
A_{N}^{\pi }\mathbf{\ }\text{is a subalgebra of }A_{N}.  \label{invar}
\end{equation}%
The latter condition holds in all the examples considered in Section \ref%
{sec4} below.

Let $\Omega ^{m}(X^{N})$ be the space of smooth differential $m$-forms on $%
X^{N}$, and
\begin{equation*}
d_{X^{N}}^{m}:L^{2}\Omega ^{m}(X^{N})\rightarrow L^{2}\Omega ^{m+1}(X^{N}),\
\mathrm{Dom}(d_{X^{N}}^{m})=\Omega ^{m}(X^{N})
\end{equation*}%
be the Hodge differential on $X^{N}$. It can be extended to an operator
\begin{equation*}
\mathbf{d}_{N}^{m}:=d_{X^{N}}^{m}\otimes 1_{A_{N}}:L^{2}\Omega
^{m}(X^{N})\otimes A_{N}\rightarrow L^{2}\Omega ^{m+1}(X^{N})\otimes A_{N},
\end{equation*}%
where $1_{A_{N}}$ is the identity operator on $A_{N}$.

\begin{proposition}
\label{d_inv} \hfill

\begin{itemize}
\item[(i)] The exterior product of two $U$-invariant $A_{N}$-valued forms is
also $U$-invariant.

\item[(ii)] The differential $\mathbf{d}_{N}^{m}$ preserves the space of $U$%
-invariant forms, that is,
\begin{equation*}
\mathbf{d}_{N}^{m}\ \left( P_{U}^{(m)}\left( \Omega ^{m}(X^{N})\otimes
A_{N}\right) \right) \subset P_{U}^{(m+1)}\left( \Omega ^{m+1}(X^{N})\otimes
A_{N}\right) .
\end{equation*}
\end{itemize}
\end{proposition}

\begin{proof}
\textrm{(i)} The statement follows from (\ref{proj_formulae}) and condition (%
\ref{invar})\textbf{.}

\textrm{(ii)} The statement follows from the fact that $\mathbf{d}_{X^{N}}^{m}$\
commutes with operators \newline $\mathfrak{S(}\phi (b_{i}))\otimes id$ and $%
id\otimes \pi (b_{i})$, $i=1,\dots ,N-1$.
\end{proof}

Let us consider the Hodge-de Rham Laplacian $H_{X^{N}}^{(m)}$ and define
\begin{equation*}
\mathbf{H}^{(m,N)}\mathbf{:}=H_{X^{N}}^{(m)}\otimes 1_{A_{N}}.
\end{equation*}%
$\mathbf{H}^{(m,N)}$ is a selfajoint operator in $L^{2}\Omega
^{m}(X^{N})\otimes A_{N}$. It follows from the definition of the operators $%
\mathbf{d}_{N}^{m}$ and $\mathbf{H}^{(m,N)}$ that%
\begin{equation*}
\mathbf{H}^{(m,N)}=\mathbf{d}_{N}^{m}\left( \mathbf{d}_{N}^{m}\right) ^{\ast
}+\left( \mathbf{d}_{N}^{m-1}\right) \mathbf{d}_{N}^{m-1}.
\end{equation*}%
Proposition \ref{d_inv} implies that $\mathbf{H}^{(m,N)}$ preserves the
space of $U$-invariant forms. Moreover, $\mathbf{H}^{(m,N)}\mathbf{\ }$%
commutes with the operators $U(b)$, $b\in B_{N}$.

Let us consider the projection
\begin{equation*}
P_{Har}^{(m,N)}:L^{2}\Omega ^{m}(X^{N})\rightarrow \mathcal{H}^{m}(X^{N}).
\end{equation*}%
Introduce the space
\begin{equation*}
\mathcal{H}_{U}^{m}(X^{N},A_{N}):=\left( \mathcal{H}^{m}(X^{N})\otimes
A_{N}\right) _{U},
\end{equation*}%
of $U$-invariant elements of $\mathcal{H}^{m}(X^{N})\otimes A_{N}$. Observe
that the projection
\begin{equation*}
P_{Har}^{(m,N)}\otimes 1_{A_{N}}:L^{2}\Omega ^{m}(X^{N})\otimes
A_{N}\rightarrow \mathcal{H}^{m}(X^{N})\otimes A_{N}
\end{equation*}%
commutes with the projection
\begin{equation*}
P_{U}^{(m,N)}:L^{2}\Omega ^{m}(X^{N})\otimes A_{N}\rightarrow L^{2}\Omega
_{\pi }^{m}(X^{N},A_{N}).
\end{equation*}%
Therefore the operator
\begin{equation*}
\mathbf{P}^{(m,N)}:=(P_{Har}^{(m,N)}\otimes
1_{A_{N}})P_{U}^{(m,N)}:L^{2}\Omega ^{m}(X^{N})\otimes A_{N}\rightarrow
\mathcal{H}_{U}^{m}(X^{N},A_{N}).
\end{equation*}%
is the projection, too.

Our goal is to compute the $L^{2}$-dimension of the space $\mathcal{H}%
_{U}^{m}(X^{N},A_{N})$, that is, the von Neumann trace of the projection $%
\mathbf{P}^{(m,N)}$. Observe that $P_{Har}^{(m,N)}\in M_{m}(X^{N})$, which
implies that $P_{Har}^{(m,N)}\otimes 1_{A_{N}}\in Q^{(m,N)}$\textbf{.}

\begin{proposition}
The factor $M_{m}(X^{N})$ does not contain the operators $\mathfrak{S}(\phi
(b))$, $b\in B_{N}$, $b\notin B_{N}^{pure}$.
\end{proposition}

\begin{proof}
Note that the space $L^{2}\Omega ^{m}(X^{N})$ can be represented in the form
\begin{equation*}
L^{2}\Omega ^{m}(X^{N})\mathbb{=}L^{2}(X)^{\otimes N}\otimes \Lambda ^{m}(%
\mathbb{C}^{N}),
\end{equation*}%
where $\Lambda ^{m}(\mathbb{C}^{N}\mathbb{)}$ is a $m$-th direct summand of
the exterior algebra $\Lambda (\mathbb{C}^{N}\mathbb{)}$. Then the factor $%
M_{m}(X^{N})$ has the form $M_{m}(X^{N})=M^{\otimes N}\otimes \mathrm{End}%
(\Lambda ^{m}(\mathbb{C}^{N}\mathbb{))}$, where $M$ is the commutant of the
action of the group $G$ in $L^{2}(X)$. The representation $\mathfrak{S}$ of
the group $S_{N}$ can be decomposed into the product $\mathfrak{S}=\mathcal{S%
}\otimes \mathcal{S}^{\prime }$, where $\mathcal{S}$ and $\mathcal{S}%
^{\prime }$ are natural actions of $S_{N}$ by permutations in $%
L^{2}(X)^{\otimes N}$ and $\Lambda ^{m}(\mathbb{C}^{N}\mathbb{)}$,
respectively (cf. (\ref{permut}) and (\ref{permut1})). It follows from
Proposition \ref{condition1}that\textbf{\ }$\mathcal{S}(\sigma )\notin
M^{\otimes N}$\textbf{, }$\sigma \neq e$\textbf{.} Therefore $\mathfrak{S}%
(\sigma )\notin M_{m}(X^{N})$, $\sigma \neq e$, which implies the result.
\end{proof}

Thus Condition \ref{outer} is satisfied, and we can apply the results of
Section \ref{sec1}. Let
\begin{equation*}
Q_{P}^{(m,N)}:=\{Q^{(m,N)},P_{U}^{(m,N)}\}^{\prime \prime }
\end{equation*}%
be the minimal von Neumann algebra containing $Q^{(m,N)}$ and $P_{U}^{(m,N)}$%
.

\begin{theorem}\label{Th3.2}
The von Neumann algebras $Q_{P}^{(m,N)}$, $\left( M_{m}(X^N)\rtimes S_{N}\right)
\otimes \mathrm{End}\ (A_{N})$ \newline and $\{Q^{(m,N)},\{U(b)\}_{b\in
B_{N}}\}^{\prime \prime }$ are isomorphic.
\end{theorem}

\begin{proof}
The first isomorphism of can be proved by the arguments similar to the proof
of Theorem \ref{equal}. The second isomorphism follows from Lemma \ref%
{braidfactor}.
\end{proof}

Theorem \ref{Th3.2} limplies that the von Neumann algebra $Q_{P}^{(m,N)}$ is a
$\mathrm{II}_{\infty }$ factor (since $\left( M_{m}(X^N)\rtimes S_{N}\right)
\otimes \mathrm{End}\ (A_{N})$ is so).

\begin{theorem}
\label{trHarm} The projection $\mathbf{P}^{(m,N)}$ belongs to $Q_{P}^{(m,N)}$
and its trace is given by the formula
\begin{equation}
\mathrm{Tr}_{Q_{P}^{(m,N)}}(\mathbf{P}^{(m,N)})=\frac{\mathrm{Tr}_{\mathrm{%
End(}{\scriptsize A}_{N})}(p_{\pi })}{N!}\sum_{m_{1}+\dots +m_{N}=m}\beta
_{m_{1}}\dots \beta _{m_{N}}  \label{BettiBN}
\end{equation}%
Here $\beta _{m}$, $m=0,1,\dots ,\dim X$, are the $L^{2}$-Betti numbers of $%
X $ and $p_{\pi }$ is the projection onto the subalgebra $A_{N}^{\pi }$ of $%
B_{N}^{pure}$-invariant elements of $A_{N}$.
\end{theorem}

\begin{proof}
The inclusion $P^{(m,N)}\in Q_{P}^{(m,N)}$ follows from the definition of
the von Neumann algebra $Q_{P}^{(m,N)}$. The equality (\ref{BettiBN})
follows from Corollary \ref{trace11} with $A\otimes B=P_{Har}^{(m,N)}\otimes
1_{A_{N}}$ and the K\"{u}nneth formula
\begin{equation}
\mathcal{H}^{m}(X^{N})=\bigoplus_{m_{1}+\dots +m_{N}=m}\ \mathcal{H}%
^{m_{1}}(X)\otimes \dots \otimes \mathcal{H}^{m_{N}}(X).  \label{Ku}
\end{equation}
\end{proof}

We will use the notation $\mathbf{b}_{m}(X^{N})=\mathrm{Tr}_{Q_{P}^{(m,N)}}(%
\mathbf{P}^{(m,N)})$ and call $\mathbf{b}_{m}(X^{N})$ the $m$-th \textit{%
braided} $L^{2}$-\textit{Betti number of} $X^{N}$.

Let us introduce the constant
\begin{equation*}
C_{N}^{\pi }=\frac{\mathrm{tr}_{{\scriptsize \mathrm{End}(A_{N})}}(p_{\pi })%
}{N!}.
\end{equation*}%
Observe that the trace on the algebra $\mathrm{End}(A_{N})$ is normalized,
that is, $\mathrm{tr}_{{\scriptsize \mathrm{End}(A_{N})}}(1_{A_{N}})=1$,
which implies that
\begin{equation}
C_{N}^{\pi }=\frac{\mathrm{dim}~A_{N}^{\pi }}{N!~\mathrm{dim}~A_{N}}\leq
\frac{1}{N!}.  \label{betti0}
\end{equation}%
Then formula (\ref{BettiBN}) can be rewritten in the form
\begin{equation}
\mathbf{b}_{m}(X^{N})=C_{N}^{\pi }~\mathrm{Tr}_{M^{(m)}}(P_{Har}^{(m,N)}).
\label{betti1}
\end{equation}

\begin{remark}
In the situation where the representation $\pi $ reduces to a representation of the
symmetric group $S_N$ (i.e. $\pi (b_{i}^{2})=1$) the projection $p_{\pi }=1$ and
formula (\ref{BettiBN}) can be rewritten in a form
\begin{equation*}
\mathrm{Tr}_{Q_{\pi }^{m}}(P_{Har}^{(m,N)})=\frac{1}{N!}\mathrm{Tr}%
_{M^{(m)}}(P_{Har}^{(m,N)})
\end{equation*}%
In this case the braided $L^{2}$-Betti numbers $b_{m}$ do not depend on the
algebra $A_{N}$ and coincide with the $L^{2}$-Betti numbers of the space of $%
N$-point configurations $\Gamma _{N}(X)$ of $X$ (\cite{DKal}, \cite{DKal1}).
\end{remark}

\section{Examples}

\label{sec4}

In this section we consider examples associated with particular
representations of the braid group $B_{N}$. We set $A_{N}=(\mathbb{C}^{2}%
\mathbb{)}^{\otimes N}$ ($\mathbb{=C}^{2^{N}}$) equipped with the natural
Abelian algebra structure.

An important class of representations $\pi $ can be constructed in the
following way. Let $C$ be a complex unitary $4\times 4$-matrix satisfying
the braid equation:
\begin{equation}
(C\otimes 1)(1\otimes C)(C\otimes 1)=(1\otimes C)(C\otimes 1)(1\otimes C),
\label{braideq}
\end{equation}%
where $1$ denotes the identity $2\times 2$-matrix. For all generators $%
b_{i}\in B_{N}$ define an operator $\pi (b_{i})$ acting in $A_{N}$ by the
formula
\begin{equation}
\pi (b_{i}):=1\otimes 1\otimes \dots \otimes C\otimes \dots \otimes 1,
\label{braideq1}
\end{equation}%
where $C$ acts on the $i$-th and $i+1$-th components of the tensor product $(%
\mathbb{C}^{2}\mathbb{)^{\otimes N}}$. Then $\pi $ defines a unitary
representation of $B_{N}$. This representation is not reduced to a
representation of $S_{N}$ if and only if $C^{2}\neq 1$.

\begin{remark}
Representations (\ref{braideq1}) are extensively studied. It is known that
the solutions of the equation (\ref{braideq}) are in one-to-one
correspondence with constant solutions $R$ of the Yang-Baxter equation
\begin{equation}
R_{j_{1}j_{2}}^{k_{1}k_{2}}R_{k_{1}j_{3}}^{l_{1}k_{3}}R_{k_{2}k_{3}}^{l_{2}l_{3}}=R_{j_{2}j_{3}}^{k_{2}k_{3}}R_{j_{1}k_{3}}^{k_{1}l_{3}}R_{k_{1}k_{2}}^{l_{1}l_{2}}.
\label{YB}
\end{equation}%
More precisely, $C=R~\Sigma $, where $\Sigma $ is a numerical matrix with
elements $\Sigma _{kl}^{ij}=\delta _{k}^{j}\delta _{l}^{i}$ and $\delta
_{k}^{j}$ denotes the Kr\"{o}neker symbol. All two-dimensional solutions of (%
\ref{YB}) have been classified in \cite{Hietar}, which gives the complete
list of unitary $4\times 4$-matrices satisfying the braid equation (\ref%
{braideq})\textbf{.}
\end{remark}

\bigskip\noindent
\textbf{Example 1.} Let $C$ be an arbitrary unitary solution of the equation
(\ref{braideq}) and $C^{2}=1$. For instance,
\begin{equation*}
C=\left(
\begin{array}{cccc}
0 & 0 & 0 & q \\
0 & \varepsilon & 0 & 0 \\
0 & 0 & \varepsilon & 0 \\
\overline{q} & 0 & 0 & 0%
\end{array}%
\right) ,
\end{equation*}%
where $\varepsilon =\pm 1$ and $|q|=1$. Then the representation $\pi $ is
reduced to a representation of $S_{N}$ , so that the corresponding
representation $\tilde{\pi}$ of the pure braid group $B_{N}^{pure}$ is
trivial ($\tilde{\pi}(b)=id$ for any $b\in B_{N}^{pure}$) and thus $%
A_{N}^{\pi }=A_{N}$. Then $C_{N}^{\pi }=\frac{1}{N!}$ and
\begin{equation}
\mathbf{b}_{m}(X^{N})=\frac{1}{N!}\sum_{m_{1}+\dots +m_{N}=m}\beta
_{m_{1}}\dots \beta _{m_{N}}.  \label{Bettiex0}
\end{equation}

\bigskip\noindent
\textbf{Example 2.}\label{RBN2} Let $q\in \mathbb{C}$ be
such that $|q|=1$ and $q^{2}\neq 1$. Set
\begin{equation*}
C=\left(
\begin{array}{cccc}
0 & 0 & 0 & q \\
0 & 1 & 0 & 0 \\
0 & 0 & 1 & 0 \\
q & 0 & 0 & 0%
\end{array}%
\right) ,
\end{equation*}%
Then $C$ is unitary solution of (\ref{braideq}) and $C^{2}\neq 1$.\textrm{\ }
\textrm{\ }

\begin{proposition}
$\mathrm{dim}~A_{N}^{\pi }=2$ for any $N\in \mathbb{N}$.
\end{proposition}

\begin{proof}
A direct calculation using formula (\ref{braideq1}).
\end{proof}

We can apply the results of Section \ref{sec3} and compute the braided $%
L^{2} $-Betti numbers $\mathbf{b}_{m}(X^{N})$ of a connected cocompact
Riemannian manifold $X$. Formula (\ref{betti0}) implies that $C_{N}^{\pi }=%
\frac{1}{2^{N-1}N!}$. Therefore%
\begin{equation}
\mathbf{b}_{m}(X^{N})=\frac{1}{2^{N-1}N!}\sum_{m_{1}+\dots +m_{N}=m}\beta
_{m_{1}}\dots \beta _{m_{N}},  \label{Bettiex1}
\end{equation}%
where $\beta _{i}$, $i=0,1,2$, are the $L^{2}$-Betti numbers of $X$.

It this example, we can give an explicit description of the braided $N$%
-particle space. A direct calculation shows that the only non-zero
components of $f=\left( f_{1},\dots ,f_{2^{N}}\right) \in \left( H\otimes
V\right) _{U}$ are $f_{a_{N-1}}$ and $f_{a_{N}}$, where the sequence $a_{N}$
is defined recursively as follows: $a_{1}=2$ and $a_{N}=2^{N}-a_{N-1}+1$.
Moreover $f_{a_{N-1}}$ and $f_{a_{N}}$ are symmetric elements, so that $%
\left( H\otimes V\right) _{U}=H_{0}^{\hat{\otimes}N}\otimes \mathbb{C}^{2}$,
where $\hat{\otimes}$ denotes the symmetric tensor product.

\bigskip\noindent
\textbf{Example 3.}\label{dimN} Let $q\in \mathbb{C}$ be such that $|q|=1$ and $q^{2}\neq 1$. Set
\begin{equation*}
C=\left(
\begin{array}{cccc}
q & 0 & 0 & 0 \\
0 & 0 & 1 & 0 \\
0 & 1 & 0 & 0 \\
0 & 0 & 0 & 1%
\end{array}%
\right) ,
\end{equation*}%
Then $C$ is unitary solution of (\ref{braideq}) and $C^{2}\neq 1$.\textrm{\ }

\begin{proposition}
\textrm{\label{kill1} $dim\ A_{N}^{\pi }=N+1$. }
\end{proposition}

\begin{proof}
Let $e_{0}=(1,0)^{t}$, $e_{1}=(0,1)^{t}$ be the canonical basis in $\mathbb{C%
}^{2}$. Denote by $e_{i_{1}\dots i_{n}}=e_{i_{1}}\otimes \dots \otimes
e_{i_{N}}$ the corresponding basis in $\mathbb{C}^{2^{N}}$.

Consider first the case of $N=2$. We have $B_{2}=\left\{ b\right\} $ and $%
\pi (b)e_{00}=qe_{00}$, $\pi (b)e_{01}=e_{10}$, $\pi (b)e_{10}=e_{01}$, $\pi
(b)e_{11}=e_{11}$. Therefore the basis of $A_{N}^{\pi }$ consists of $e_{01}$%
, $e_{10}$ and $e_{11}$, and $\mathrm{dim}A_{N}^{\pi }=3$.

A direct check shows that $e_{\underset{N}{\underbrace{11\dots 11}}}$ and $%
e_{\underset{N}{\underbrace{11\dots 10}}}$ belong to $A_{N}^{\pi }$. Let us
prove that for any multiindex $I=i_{1}\dots i_{N}$, $i_{k}\in \{0,1\}$, such
that $e_{I}\in A_{N}^{\pi }$ the equality $i_{N}=0$ implies that $%
i_{1}=\dots i_{N-1}=1$. Indeed, let $i_{N}=0$ and $k:=\mathrm{max}_{j\in
\{1,2,\dots ,N-1\}}\{i_{j}=0\}>0$. Then it is easy to see that
\begin{eqnarray*}
e_{I}=\pi (x_{k,N-1})e_{I} &=&\pi (b_{N-1}b_{N-2}\dots
b_{k+1}b_{k}^{2}b_{k+1}^{-1}\dots b_{N-2}^{-1}b_{N-1}^{-1})e_{I} \\
&=&\pi (b_{N-1}b_{N-2}\dots b_{k+1}b_{k}^{2})e_{i_{1}\dots i_{k-1}001\dots
1}=q^{2}e_{I},
\end{eqnarray*}%
because the space $A_{N}^{\pi }$ is invariant under the representation  $\pi$
and $\pi (b_{k}^{2})e_{i_{1}\dots i_{k-1}001\dots
1}=q^{2}e_{i_{1}\dots i_{k-1}001\dots 1}$ (recall that $\pi (b_{i})=1\otimes
\dots \otimes 1\otimes C\otimes 1\otimes \dots \otimes 1$). Therefore $%
e_{I}=0$.

Let us prove that $e_{I}\in A_{N}^{\pi }$ iff $e_{I1}\in A_{N+1}^{\pi }$.
Indeed, let $e_{I}\in A_{N}^{\pi }$. To prove that $e_{I1}\in A_{N+1}^{\pi }$
it suffice to check that $\pi (x_{k,N})e_{I1}=e_{I1}$ for all $k\leq N$. If $%
i_{N}=1$ we have $\pi (x_{k,N})e_{I1}=\pi (b_{N}x_{N-1,k})e_{I1}=\pi
(b_{N})e_{I1}=e_{I1}$ since $\pi (x_{k,N-1})e_{I}=e_{I}$. If $i_{N}=0$ then
by previous statement we have that $i_{1}=\dots =i_{N-1}=1$ and $\pi
(x_{k,N})e_{I1}=\pi (b_{N}x_{k,N-1})e_{11\dots 10}=\pi (b_{N})e_{11\dots
10}=e_{11\dots 101}=e_{I1}$.

The equalities $\pi (x_{i,j})e_{I1}=e_{I1}$, $i\leq j\leq N-1$ and the fact
that the operators $\pi (x_{i,j})$ are acting only in first $N-1$ indexes
imply that $e_{I}\in A_{N}^{\pi }$ provided $e_{I1}\in A_{N+1}^{\pi }$. This
statement combined with the induction arguments and the fact that $e_{%
\underset{N+1}{\underbrace{11\dots 10}}}\in A_{N+1}^{\pi }$\ complete the
proof.
\end{proof}

It follows from Proposition \ref{kill1} that $C_{N}^{\pi }=\frac{N+1}{2^{N}N!%
}$. Therefore the braided $L^{2}$-Betti numbers $\mathbf{b}_{m}(X^{N})$ of $%
X $ have the following form:
\begin{equation}
\mathbf{b}_{m}(X^{N})=\frac{N+1}{2^{N}N!}\sum_{m_{1}+\dots +m_{N}=m}\beta
_{m_{1}}\dots \beta _{m_{N}},  \label{Bettiex2}
\end{equation}%
where $\beta _{i}$, $i=0,1,2$, are the $L^{2}$-Betti numbers of $X$.

The structure of the corresponding braided $N$-particle space is more
complicated then is previous example. It is not hard to see that the
representation $\tilde{\pi}$ of the symmetric group in the space $A_{N}^{\pi
}$ that corresponds to the representation $\pi $ can be decomposed into a
direct sum of two cyclic representations $\tilde{\pi}_{1}$ and $\tilde{\pi}%
_{2}$. Representation $\tilde{\pi}_{1}$ is trivial and acts in
one-dimensional space generated by the basis element $e_{11\dots 1}$ while
representation $\tilde{\pi}_{2}$ permutetes the other $N$ basis elements in $%
A_{N}^{\pi }$. Therefore the space $H^{\otimes _{\pi }N}$ can be decomposed
into the direct sum of the space $L^{2}(X)^{\hat{\otimes}N}$ of symmetric
functions and the space $H^{\otimes _{\tilde{\pi}_{2}}N}$ of elements from $%
L^{2}(X^{N})\otimes \mathbb{C}^{N}$ that are invariant under representation $%
\tilde{U}_{N}$ of the group $S_{N}$ given by the relation
\begin{equation*}
(\tilde{U}_{N}(\sigma _{i})f)(x_{1},\dots ,x_{i},x_{i+1},\dots x_{N})=\tilde{%
\pi}_{2}(\sigma _{i})f(x_{1},\dots ,x_{i+1},x_{i},\dots ,x_{N}),
\end{equation*}%
where $\sigma _{1},\dots ,\sigma _{N-1}$ are generating elements of the
group $S_{N}$.

\bigskip\noindent
\textbf{Example 4.}\label{dim0} Let $q\in \mathbb{C}$ be such that $|q|=1$ and $q^{2}\neq 1$
\begin{equation*}
C=\left(
\begin{array}{cccc}
0 & 0 & 0 & 1 \\
0 & q & 0 & 0 \\
0 & 0 & q & 0 \\
1 & 0 & 0 & 0%
\end{array}%
\right) .
\end{equation*}%
Then $C$ is unitary solution of (\ref%
{braideq}) and $C^{2}\neq 1$.\textrm{\ }

\begin{proposition}
$\mathrm{dim}\ A_{N}^{\pi }=0$.
\end{proposition}

\begin{proof}
Similar to the Example 2 we can show that $\pi (b_{k}^{2})e_{i_{1}\dots
i_{k}i_{k+1}\dots i_{N}}=q^{2}e_{i_{1}\dots i_{N}}$ iff $i_{k}i_{k+1}=01$ or
$i_{k}i_{k+1}=10$. Therefore any vector $v\in A_{N}^{\pi }$ should have the
form $v=\alpha e_{00\dots 0}+\beta e_{11\dots 1}$. On the other hand, $\pi
(x_{1,2})e_{00\dots 0}=q^{2}e_{00\dots 0}$ and $\pi (x_{1,2})e_{11\dots
1}=q^{2}e_{11\dots 1}$, so that $\pi (x_{1,2})v=q^{2}v$, which implies that $%
v=0$. Therefore $\mathrm{dim}\ A_{N}^{\pi }=0$.
\end{proof}

Thus, $\mathbf{b}_{k}(X^{N})=0$.

\section{Fock space of braid-invariant harmonic forms: $L^{2}$-dimensions
and index}

\label{sec5}

Let us consider the infinite disjoint union
\begin{equation*}
\mathfrak{X=}\bigsqcup\limits_{N=1}^{\infty }X^{N}.
\end{equation*}%
The space $L^{2}\Omega ^{m}(\mathfrak{X})$ of square-integrable $m$-forms on
$\mathfrak{X}$ has the form%
\begin{equation}
L^{2}\Omega ^{m}(\mathfrak{X})=\bigoplus\limits_{N}L^{2}\Omega ^{m}(X^{N}).
\label{forms-decomp}
\end{equation}%
We can define the Hodge-de Rham Laplacian $H_{\mathfrak{X}}^{(m)}$ on $%
L^{2}\Omega ^{m}(\mathfrak{X})$ component-wise, that is,
\begin{equation*}
H_{\mathfrak{X}}^{(m)}:=\sum_{N}H_{X^{N}}^{(m)}.
\end{equation*}%
For the corresponding spaces of harmonic forms we obviously have the
decomposition%
\begin{equation}
\mathcal{H}^{m}(\mathfrak{X}):=\mathrm{Ker}~H_{\mathfrak{X}%
}^{(m)}=\bigoplus\limits_{N}\mathcal{H}^{m}(X^{N}).  \label{harmonic-decomp}
\end{equation}%
Formulae (\ref{forms-decomp}) and (\ref{harmonic-decomp}) motivate us to
\textit{define} the space
\begin{equation*}
L^{2}\Omega _{U}^{m}(\mathfrak{X,}A):=\bigoplus\limits_{N}^{{}}L^{2}\Omega
_{U}^{m}(X^{N},A_{N}),
\end{equation*}%
which can be regarded as the space of "$m$-\textit{forms on the braided
configuration space}". Similar to the case of the de Rham complex (\ref%
{forms-decomp}), the Hodge-de Rham Laplacian $\mathbf{H}^{(m)}$ on $L_{\pi
}^{2}\Omega ^{m}(\mathfrak{X})$ can be defined component-wise,%
\begin{equation*}
\mathbf{H}^{(m)}:=\sum_{N}\mathbf{H}^{(m,N)}.
\end{equation*}%
The corresponding space of $U$-invariant harmonic $m$-forms can be decomposed into the
direct sum%
\begin{equation*}
\mathcal{H}_{U}^{m}(\mathfrak{X,}A):=\mathrm{Ker}~\mathbf{H}%
^{(m)}=\bigoplus\limits_{N}\mathcal{H}_{U}^{m}(X^{N},A_{N}).
\end{equation*}%
Let $\mathbf{P:}=\sum_{m}\mathbf{P}^{(m)}$, where $\mathbf{P}^{(m)}$ is the
projection onto $\mathcal{H}_{U}^{m}(\mathfrak{X,}A)$. Moreover, we have the equality $\mathbf{P}%
^{(m)}=\sum_{N}\mathbf{P}^{(m,N)}$, which implies that $\mathbf{P}^{(m)}$  an element of the von Neumann
algebra $Q_{P}^{(m)}:=\bigoplus\limits_{N}Q_{P}^{(m,N)}$.

We introduce a regularized dimension $\mathbf{b}_{m}(\mathfrak{X)}$ of the
space $\mathcal{H}_{U}^{m}(\mathfrak{X,}A)$ by the formula%
\begin{equation*}
\mathbf{b}_{m}(\mathfrak{X)}=\sum_{N=0}^{\infty }\mathbf{b}_{m}^{(N)}=%
\mathrm{Tr}_{Q_{P}^{(m)}}\mathbb{~}\mathbf{P}^{(m)}
\end{equation*}%
and define the \textit{supertrace}
\begin{equation}
\mathbb{STR~}\mathbf{P=}\sum_{m=0}^{\infty }(-1)^{m}\mathbf{b}_{m}(\mathfrak{%
X}).  \label{ind}
\end{equation}%
Here we use the convention $\mathbf{b}_{0}=1$.

\begin{theorem}
\label{index-theorem} 
The series in the right-hand side of (\ref{ind}) converges absolutely and%
\begin{equation*}
\mathbb{STR~}\mathbf{P}=\sum_{N}C_{N}^{\pi }\chi (K)^{N}<\infty ,
\end{equation*}%
where $\chi (K)$ is the Euler characteristic of $K$.
\end{theorem}

\begin{proof}
Formula (\ref{betti0}) shows that $C_{N}^{\pi }\leq \frac{1}{N!}.$ Then
\begin{eqnarray*}
\sum_{m=0}^{\infty }\mathbf{b}_{m}(\mathfrak{X)} &\leq &\sum_{m=0}^{\infty
}\sum_{N=0}^{\infty }\frac{1}{N!}\sum_{k_{1}+k_{2}+...+k_{N}=m}\beta
_{k_{1}}\dots \beta _{k_{N}} \\
&=&\sum_{N=0}^{\infty }\frac{1}{N!}\left( \sum_{k}\beta _{k}\right)
^{N}<\infty .
\end{eqnarray*}%
Then
\begin{eqnarray*}
\mathbb{STR~}\mathbf{P} &=&\sum_{m=0}^{\infty }\left( -1\right) ^{n}\mathbf{b%
}_{m}(\mathfrak{X)} \\
&=&\sum_{m=0}^{\infty }\left( -1\right) ^{m}\sum_{N}C_{N}^{\pi
}\sum_{k_{1}+k_{2}+...+k_{N}=m}\beta _{k_{1}}\dots \beta _{k_{N}} \\
&=&\sum_{N}C_{N}^{\pi }\sum_{k_{1},k_{2},...,k_{N}}\left( -1\right)
^{k_{1}}\beta _{k_{1}}...\left( -1\right) ^{k_{N}}\beta _{k_{N}} \\
&=&\sum_{N}C_{N}^{\pi }\left( \sum_{k}\left( -1\right) ^{k}\beta _{k}\right)
^{N} \\
&=&\sum_{N}C_{N}^{\pi }\chi (K)^{N}
\end{eqnarray*}%
because of the equality $\sum_{m}\left( -1\right) ^{m}\beta _{m}=\chi (K)$
(the $L^{2}$ index theorem \cite{Ati}).
\end{proof}

Let us consider the examples of the previous section. In all these examples we have
$A_{N}=\left( \mathbb{C}^{2}\right) ^{\otimes N}$ so that $\dim A_{N}=2^{N}$%
. Thus, we can apply Theorem \ref{index-theorem} and compute $\mathbb{STR~}%
\mathbf{P}_{\pi }$.

\bigskip\noindent
\textbf{Example 1 }(revisited). $C_{N}^{\pi }=\frac{1}{N!}$ and $\mathbb{STR~%
}\mathbf{P}=e^{-\chi (K)}$. The latter expression coincides with the formula
derived in \cite{AD} for the case the de Rham complex over the configuration
space equipped with the Poisson measure.

\bigskip\noindent
\textbf{Example 2\ }(revisited). $C_{N}^{\pi }=\frac{1}{2^{N-1}N!}.$ Then%
\begin{equation*}
\mathbb{STR~}\mathbf{P}=\sum_{N}(-1)^{N}C_{N}^{\pi }\chi
(K)^{N}=\sum_{N}(-1)^{N}\frac{2}{N!}\left( \frac{\chi (K)}{2}\right)
^{N}=2e^{-\frac{\chi (K)}{2}}.
\end{equation*}

\bigskip\noindent
\textbf{Example 3} (revisited). $C_{N}^{\pi }=\frac{N+1}{2^{N}N!}.$ Then%
\begin{eqnarray*}
\mathbb{STR~}\mathbf{P} &=&\sum_{N}(-1)^{N}C_{N}^{\pi }\chi (K)^{N} \\
&=&\sum_{N}(-1)^{N}\frac{N}{N!}\left( \frac{\chi (K)}{2}\right)
^{N}+\sum_{N}(-1)^{N}\frac{1}{N!}\left( \frac{\chi (K)}{2}\right) ^{N} \\
&=&-\frac{\chi (K)}{2}\sum_{N}(-1)^{N-1}\frac{1}{\left( N-1\right) !}\left(
\frac{\chi (K)}{2}\right) ^{N-1}+\sum_{N}(-1)^{N}\frac{1}{N!}\left( \frac{%
\chi (K)}{2}\right) ^{N} \\
&=&1-\frac{\chi (K)}{2}e^{-\frac{\chi (K)}{2}}+e^{-\frac{\chi (K)}{2}}.
\end{eqnarray*}

\begin{remark}
Under the additional assumption of $\mathrm{dim}X=2$, the only non-zero $%
L^{2}$-Betti number of $X$ is $\beta =\beta _{1}(X)$. Then, according to
formula (\ref{BettiBN}), the only non-zero braided $L^{2}$-Betti number of $%
X^{N}$ is $\mathbf{b}_{N}(X^{N})=C_{N}^{\pi }\beta ^{N}$, and therefore
\begin{equation*}
\mathbf{b}_{m}(\mathfrak{X})=C_{N}^{\pi }\beta ^{m},\ m=1,2,3,....
\end{equation*}
\end{remark}

\begin{remark}
The right-hand side of formula (\ref{ind}) can be understood as a
regularized index of the Dirac operator $D+D^{\ast }$, where $%
D:=\sum_{m}\sum_{N}\mathbf{d}_{N}^{m}$, see \cite{Ati} and e.g. \cite{Mat}
for the discussion of von Neumann supertraces in geometry and topology of
Riemannian manifolds and their relation to $L^{2}$-index theorems, and \cite%
{AD}, \cite{ADKal1}, \cite{ADKal2} for the extension of these notions to the
framework of infinite configuration spaces.
\end{remark}

\section*{Acknowledgments}\label{sec:Ack}

The authors would like to thank Sergio Albeverio, Yuri Kondratiev,
Eugene Lytvynov and Leonid Vainerman for
helpful discussions.

\end{document}